\documentclass{amsart}

\usepackage{inputenc}
\usepackage{amsmath}
\usepackage{amsthm}
\usepackage{amsfonts}
\usepackage{amssymb}

\theoremstyle{definition}
\newtheorem{defin}{Definition}[section]

\theoremstyle{definition}
\newtheorem{teo}{Theorem}
\newtheorem{cor}{Corollary}[section]
\newtheorem{prop}{Proposition}[section]
\newtheorem{lema}{Lemma}[section]

\theoremstyle{remark}
\newtheorem{rem}{Remark}[section]

\begin{document}	
\title[Invariant Surfaces for Toric Type Foliations]
{Invariant Surfaces for Toric Type Foliations in Dimension Three}
\author{Felipe Cano}
\address{Dpto. \'Algebra, An\'alisis Matem\'atico, Geometr\'ia y Topolog\'ia, Facultad de Ciencias, Universidad de Valladolid, Paseo de Bel\'en, 7, 47011 Valladolid, Spain}
\email{fcano@agt.uva.es}
\author{Beatriz Molina-Samper}
\address{Dpto. \'Algebra, An\'alisis Matem\'atico, Geometr\'ia y Topolog\'ia and Instituto de Matem\'aticas de la Universidad de Valladolid, Facultad de Ciencias, Universidad de Valladolid, Paseo de Bel\'en, 7, 47011 Valladolid, Spain}
\email{beatriz.molina@uva.es}
\thanks{Both authors are supported by the Ministerio de Econom\'ia y Competitividad from Spain, under the Project ``Algebra y geometr\'ia en sistemas din\'amicos y foliaciones singulares.'' (Ref.:  MTM2016-77642-C2-1-P). The second author is also supported by the Ministerio de Educaci\'on, Cultura y Deporte of Spain (FPU14/02653 grant)}			

\date{}
\dedicatory{}
\begin{abstract}
A foliation is of toric type when it has a combinatorial reduction of singularities. We show that every toric type foliation on $(\mathbb{C}^3,0)$, without saddle-nodes, has invariant surface. We extend the argument of Cano-Cerveau, done for the non-dicritical case, to the compact dicritical components of the exceptional divisor. These components are  projective toric surfaces and the isolated invariant branches of the induced foliation extend to closed irreducible curves. We build the invariant surface as a germ along the singular locus and those closed irreducible invariant curves. The result of Ortiz-Rosales-Voronin, about the distribution of invariant branches in dimension two, is a key argument in our proof.
\end{abstract}

\maketitle
\section{Introduction}
The problem of existence of invariant hypersurfaces for a holomorphic codimension one foliation is a leitmotif in the theory of holomorphic singular foliations, starting with a question of Ren\'e Thom. The main result in this paper is a contribution to this problem, stated as follows:
	
\begin{teo} 
\label{teo:main}
Every toric type complex hyperbolic foliation on $(\mathbb{C}^3,0)$ has an invariant surface.
\end{teo}
A foliation is of toric type when it admits a combinatorial reduction of singularities, with respect to a given normal crossings divisor. The expression ``complex hyperbolic'' means that we can not extract saddle-nodes from the foliation. Anyway, these foliations may be dicritical, in the sense that there are some generically transversal irreducible components of the exceptional divisor after reduction of singularities.
	
The existence of invariant hypersurface has a positive answer in the non-dicritical situation. The result is due to Camacho-Sad in the bidimensional case \cite{Cam-S}, to Cano-Cerveau in the three-dimensional case \cite{Can-C} and to Cano-Mattei in general ambient dimension \cite{Can-M}. In contrast to what it happens in dimension two, there are dicritical examples of codimension one foliations in dimension three without invariant surface; the first family of such examples was given by Jouanolou \cite{Jou}.
	
In order to prove the existence of an invariant surface for a dicritical foliation on $(\mathbb{C}^3,0)$, it is essential to have ``good properties'' for the restriction of the foliation to compact dicritical components after reduction of singularities. In the context of toric type foliations, we see that the compact components of the exceptional divisor are nonsingular projective toric surfaces, in the sense of Toric Geometry, endowed in a natural way with a normal crossings divisor, compatible with the ambient divisor. In a previous work \cite{Mol}, we have proved that a toric type foliation $\mathcal{G}$ on a projective toric surface $S$, with associated divisor $D$, satisfies the following ``prolongation property for isolated invariant branches'':
\begin{quote}
``Every isolated invariant branch $(\Gamma,p)$ extends to a unique closed irreducible curve $Y \subset S$; moreover, all the branches of $Y$ at the points of $Y\cap D$ are isolated.''
\end{quote}
In a general way, if we have this property in the restriction to each one of the compact dicritical components after reduction of singularities, we can extend the argument of Cano-Cerveau in \cite{Can-C} to prove the existence of an invariant surface, provided we have at least a trace type simple singularity. The details of this argument are in Subsection \ref{subs:Extended Partial Separatrices}.
	
Now, it would be enough to find a trace type singular point after reduction of singularities, in the toric type context. Such a point appears if and only if there is at least one invariant component in the exceptional divisor, as we show in Section \ref{sec:Corner Type Points under Combinatorial Blowing-ups}. In the proof of this result, we invoke a refined version of Camacho-Sad's theorem proved by Ortiz-Rosales-Voronin in \cite{Ort-R-V}.
	
The remaining cases correspond to toric type foliations where we start with only two components of the initial divisor. In this situation only blowing-ups centered in curves are allowed, in an \'etale way over an initial one, and the existence of invariant surface follows by direct arguments.

We would like to end this introduction with a grateful acknowledgment to the two referees, for their important suggestions that have considerably improved the manuscript.
\vfill

\section{Preliminaries}
We introduce basic definitions and results concerning the theory of codimension one holomorphic singular foliations (for short, foliations). Some references are  \cite{Can, Can-C, Cer-M, Mat-M}.

A \emph{foliation} $\mathcal{F}$ on a nonsingular complex analytic space $M$ of dimension $n$ is an integrable invertible coherent $\mathcal{O}_M$-submodule $\mathcal{F}\subset \Omega_M^1$, that is saturated in the sense that ${\mathcal F}^{\bot\bot}={\mathcal F}$. This means that $\mathcal F$ is
locally generated at each point $p\in M$ by a holomorphic one-form $\omega \in \Omega^1_{M,p}$ satisfying $\omega \wedge d\omega=0$, that we write in local coordinates as
$$
\omega=f_1dx_1+f_2dx_2+\cdots+f_ndx_n, \quad f_i\in \mathcal{O}_{M,p},
$$
where the coefficients $f_i$ have no common factors. The \emph{order $\nu_p(\mathcal{F})$} of $\mathcal{F}$ at $p$ is defined by
$$
\nu_p(\mathcal{F})=\nu_p(f_1,f_2,\ldots,f_n)=\min\{\nu_p(f_i); \; i=1,2,\ldots,n \},
$$
where $\nu_p(f_i)$ is the order at $p$ of the coefficient $f_i$. The \emph{singular locus} $\text{Sing}(\mathcal{F})$ is the closed analytic subset of $M$ given by the points $p\in M$ with $\nu_p(\mathcal{F})\geq 1$.
Note that the codimension of the singular locus is greater or equal than two, otherwise we should find a common factor in the coefficients of local generators.

Note that any integrable local meromorphic one-form $\eta$ defines locally a foliation, just by considering a holomorphic form $\omega=(f/g)\eta$ without common factors in its coefficients.

The \emph{dimensional type $\tau_p(\mathcal{F})$  of $\mathcal{F}$ at $p$} is defined by the fact that
$n-\tau_p(\mathcal{F})$ is the dimension of the $\mathbb{C}$-vector space spanned by the vectors $\xi(p)\in T_pM$, where $\xi$ is a germ of vector field such that $\omega(\xi)=0$. In view of the classical Frobenius' Theorem and the Rectification Theorem of non-singular vector fields (see \cite{War}, for instance) we know that there are local coordinates $\mathbf{x}=(x_1,x_2,\ldots,x_n)$ such that $\mathcal{F}$ is locally generated by a one-form $\omega$ of the type
\begin{equation}
\label{eq:escrituraminimal}
\omega = \sum\nolimits_{i=1}^{\tau}f_i(x_1,x_2,\ldots x_{\tau})dx_{i}, \quad \tau=\tau_p(\mathcal{F}).
\end{equation}
Thus, the foliation is locally an analytic cylinder over a foliation on a space of dimension $\tau$. Note that $\tau=1$ if and only if $p\not \in \text{Sing}(\mathcal{F})$.

We recall that a hypersurface $H$ of $M$ is \emph{invariant for $\mathcal{F}$} when any local equation $f$ of $H$ divides $\omega \wedge df$. By Frobenius theorem, there is a unique germ of invariant hypersuface through each nonsingular point.

In a more general way, a morphism $\phi:(N,q) \rightarrow (M,p)$ is called invariant for $\mathcal F$ when $\phi^*\omega=0$ for the local generators $\omega$ of $\mathcal F$.  A closed analytic subspace $Y\subset M$ is called \emph{invariant for $\mathcal{F}$ at $p\in Y$} if each morphism $\phi:(\mathbb{C},0) \rightarrow (M,p)$ factoring through $(Y,p)$ is invariant.  We say that $Y$ is \emph{invariant for $\mathcal{F}$} when the property holds at each point $p\in Y$. Being invariant at a point is an open and closed property on $Y$. Hence, an irreducible subspace $Y$ is invariant if and only if it is invariant at a point.

The concept of ``complex hyperbolic foliation'' can be found in \cite{Can-R}:
\begin{defin}
A foliation $\mathcal F$ on $M$ is \emph{complex hyperbolic at $p\in M$} if there is no holomorphic map $\phi:(\mathbb{C}^2,0) \rightarrow (M,p)$ such that $0$ is a saddle-node for $\phi^*\mathcal{F}$. The foliation is \emph{complex hyperbolic} if the property holds at each point. A foliation $\mathcal F$ on $M$ is \emph{strongly complex hyperbolic at $p\in M$} if for any holomorphic map $\pi: (M',\pi^{-1}(p))\rightarrow (M,p)$ obtained by composition of blowing-ups the transformed foliation $\pi^*{\mathcal F}$ is complex hyperbolic.
\end{defin}

These properties has been required by the authors of the classical paper \cite{Cam-LN-S} in their study of the so-called ``generalized curves''  in ambient dimension two.
In fact they asked to have no saddle-nodes after reduction of singularities, but this is equivalent to the properties in the above definition. In general dimension, we require to have no ``hidden saddle-nodes'' obtained by pull-back by a holomorphic morphism.

In dimension three, thanks to the existence of reduction of singularities \cite{Can}, or in the ``non-dicritical case'',  being complex hyperbolic is equivalent to being strongly complex hyperbolic.
In the general case, we do not know if both definitions are  equivalent.

An important and evident feature of the strongly complex hyperbolic condition is that it is stable under blowing-ups.

\begin{defin}
A foliation $\mathcal F$ on $M$ is \emph{dicritical at $p\in M$} when there is a map $\phi:(\mathbb{C}^2,0) \rightarrow (M,p)$, such that $\phi^{-1}\mathcal{F}=(dx=0)$ and $\phi(y=0)$ is invariant for $\mathcal{F}$.
\end{defin}

In dimensions two and three, this definition is equivalent to the fact that there is a generically transversal component of the exceptional divisor after a reduction of singularities with invariant centers. Let us note that any germ of foliation having a holomorphic first integral is non-dicritical. Indeed, consider the foliation given by $df=0$ and assume that it is dicritical.  The pull-back is given by $d(f\circ\phi)=0$, since it is the foliation $dx=0$, the function $f\circ\phi$ is of the form $f\circ\phi=\psi(x)$. Now the fact that $\phi(y=0)$ is invariant means that $\psi$ is a constant function, contradiction.

There are good properties for the foliations that are both complex hyperbolic and non-dicritical. The more significant is that a reduction of singularities of their invariant hypersurfaces is automatically a reduction of singularities of the foliation \cite{Fer-M}. Such foliations could be denominated {\em generalized hypersurfaces}. In this paper we take the complex hyperbolic hypothesis, but we allow dicritical situations.

\section{Toric Type Foliations}
The concept of toric type foliated space was introduced in \cite{Cam-C} for the bidimensional case. Here we generalize it to higher dimension.

A foliation $\mathcal{F}$ on $M$ is called of \emph{toric type} if there is a normal crossings divisor $E$ on $M$ and a combinatorial sequence of blowing-ups
$$
\sigma:(M',E') \to (M,E)
$$
providing a reduction of singularities of the \emph{foliated space} $(M,E;\mathcal{F})$. Let us explain the terms in this definition.

A \emph{combinatorial blowing-up} $\pi:(M_1,E^1)\to (M,E)$ is any blowing-up centered at $Y$, where $Y$ is the closure of a stratum of the natural stratification of $M$ induced by $E$. The divisor $E^1$ is given by
$$
E^1=\pi^{-1}(E \cup Y)=\pi^{-1}(E)
$$
(note that $Y\subset E$). For a reduction of singularities of $(M,E;\mathcal{F})$, we assume in addition that the centers are invariant for $\mathcal{F}$.

We split the divisor $E$ as $E=E_{\text{inv}} \cup E_{\text{dic}}$, where $E_{\text{inv}}$ is the union of the invariant irreducible components of $E$ and $E_{\text{dic}}$ is the union of the generically transversal ones (dicritical components). At a given point $p\in M$, we denote by $e_p(E)$ the number of irreducible components of $E$ through $p$; note that $e_p(E)=\nu_p(E)$, where $\nu_p(E)$ stands for the multiplicity at $p$, since $E$ has the normal crossings property.

The foliated space $(M,E;\mathcal{F})$ is \emph{desingularized} if each point $p \in M$ is simple.
Let us recall the concept of \emph{simple point} (see \cite{Can, Can-C}). We precise here the definition only when the foliation is complex hyperbolic. A simple point may be of ``corner type'' or of ``trace type''. We give first the definition of simple corner point.

Let us consider $p\in M$ and denote $\tau$ the dimensional type of $\mathcal{F}$ at $p$.
Assume that $\mathcal F$ is complex hyperbolic. We say that $p$ is a \emph{simple corner for $(\mathcal{F},E)$} if there is a local coordinate system $\mathbf{x}=(x_1,x_2,\ldots,x_n)$ and a local generator $\omega= (\mbox{$\prod$}_{j=1}^\tau x_j)\eta$ of $\mathcal{F}$ satisfying the following conditions:
\begin{enumerate}
\item We have $E_{\text{inv}}=\cup_{j=1}^{\tau}(x_j=0)$ and $E_{\text{dic}}\subset\cup_{j=\tau+1}^{n}(x_j=0)$.
\item We can write $\eta =\sum\nolimits_{j=1}^\tau a_j(x_1,x_2,\ldots,x_\tau)dx_j/x_j$ in such a way that $\sum_{j=1}^\tau r_ja_j(p) \ne 0$ for every non-zero $(r_1,r_2,\ldots,r_\tau)\in \mathbb{Z}_{\geq 0}^{\tau}$.
\end{enumerate}
We say that $p$ is a \emph{simple trace for $(\mathcal{F},E)$} if there is a germ $(H,p)$ of nonsingular hypersurface invariant for $\mathcal{F}$ with $H \not\subset E$ and such that $p$ is a simple corner for $(\mathcal{F},E\cup H)$.

\begin{rem} 
\label{rem:tipotraza}
The singular locus $\text{Sing}(\mathcal{F})$ around a simple corner $p$ for $(\mathcal{F},E)$ is given by the union of the two by two intersections of the irreducible components of $E_\text{inv}$ passing through $p$. Moreover, all the singularities around $p$ are simple corners. When $p$ is a simple trace point, the set $\mathcal{T}_{\mathcal{F},E}$ of trace type simple singularities around $p$ is given by
$$
\mathcal{T}_{\mathcal{F},E}=H \cap E_\text{inv},
$$
where $H$ is the invariant hypersurface passing through $p$ and not contained in $E$. 
\end{rem}

The {\em adapted singular locus $\text{Sing}(\mathcal{F},E)$} is by definition the union of the singular locus $\text{Sing}(\mathcal{F})$ and the set of non-simple points. Note that the nonsingular simple points are exactly the ones where $E$ and $\mathcal{F}$ ``have normal crossings''.

\section{Extended Partial Separatrices}
Partial separatrices have been introduced in \cite{Can-R} to formalize the arguments in \cite{Can-C} for the construction of invariant surfaces of non-dicritical foliations in ambient dimension three. We extend the concept to the dicritical case to give afterwards properties that assure the existence of invariant surfaces supported by them.

Let $(M,E;\mathcal{F})$ be a desingularized complex hyperbolic foliated space in dimension three. We assume that $M$ is a germ along a compact set $K \subset E$; this is the typical situation that we obtain when we perform finitely many blowing-ups starting with $(\mathbb{C}^3,0)$.

By Remark \ref{rem:tipotraza}, the set of trace type simple singularities $\mathcal{T}_{\mathcal{F},E}$ is a closed analytic subspace of $M$. It is the union of the irreducible components of $\text{Sing}(\mathcal{F})$ that are contained exactly in one irreducible component of $E$. Let us take the following definition coming from \cite{Can-R}:
\begin{defin}
Let $(M,E;{\mathcal F})$ be a desingularized  complex hyperbolic foliated space of dimension three. We call {\em partial separatrices } of $(M,E;{\mathcal F})$ to the connected components of the set ${\mathcal T}_{{\mathcal F},E}$.
\end{defin}

Consider a partial separatrix ${\mathcal C}\subset {\mathcal T}_{{\mathcal F}, E}$. Given a point $q$ of $\mathcal{C}$, there is a unique germ $S_{q}$ of invariant irreducible surface such that $S_q \not\subset E$; moreover, we have that $S_q\cap E_{\text{inv}}$ is the germ of $\mathcal{T}_{\mathcal{F},E}$ at $q$. Note that here we take the hypothesis of being complex hyperbolic and then the convergence of $S_q$ is assured. The above property has been introduced in \cite{Can-C} as being the essential argument to build global invariant surfaces of $(M,E;{\mathcal F})$. The construction in \cite{Can-C} works if $\mathcal C$ does not meet dicritical components of $E$. Hence, the extension argument in \cite{Can-C} assures that the surface $S_{q}$ extends to an irreducible closed surface $S_\mathcal{C}\subset M$ invariant for $\mathcal{F}$, when $\mathcal{C} \cap E_{\text{dic}}=\emptyset$. When ${\mathcal C}\cap E_{\text{dic}}\ne\emptyset$, we can assure the existence of a germ of invariant surface $(S_{\mathcal C},\mathcal{C})$, along the germification set $\mathcal C\cap K$, where $K$ is the germification set of $M$. The important remark is that the immersion of germs
$$
(S_{\mathcal C},\mathcal{C}\cap K)\subset (M,K)
$$
is not necessarily closed: for instance, this phenomena appears in the reduction of singularities of Jouanolou's examples.

We are interested in connecting the partial separatrices of $(M,E; \mathcal{F})$ through closed invariant curves contained in the dicritical components of $E$. Denote by $\Sigma$ the set whose elements are the closed irreducible curves $Z$ invariant for $\mathcal{F}$ and satisfying:
\begin{equation}
\label{eq:sigma}
Z \subset E_{\text{dic}}, \; Z \not \subset E_\text{inv}, \; Z \cap \text{Sing}(\mathcal{F})\ne \emptyset.
\end{equation}
Note that $\Sigma$ is a finite set. Let $\tilde{\Sigma}$ be the union of the $Z \in \Sigma$ and denote $\mathcal{U}_{\mathcal{F},E}=\mathcal{T}_{\mathcal{F},E} \cup \tilde{\Sigma}$.	This allows us to establish the following definition:
\begin{defin} 
Let $(M,E; {\mathcal F})$ be a desingularized complex hyperbolic foliated space of dimension three, where $M$ is a germ along a compact set $K\subset E$. The connected components of
$\mathcal{U}_{\mathcal{F},E}$ are called {\em extended partial separatrices of
$(M,E; {\mathcal F})$}.
\end{defin}
Let $\mathcal{E}$ be an extended partial separatrix of $(M,E;{\mathcal F})$. There is a unique germ $(S_\mathcal{E},\mathcal{E}\cap K)$ of invariant surface along $\mathcal{E}\cap K$ such that $S_\mathcal{E}\not \subset E$. Next definition is important in our arguments:
\begin{defin} 
Let $(M,E;{\mathcal F})$ be a desingularized complex hyperbolic foliated space of dimension three, where $M$ is a germ along a compact set $K\subset E$.  We say that an extended partial separatrix $\mathcal E$ is {\em complete} when we have that $S_{\mathcal{E}} \cap E= \mathcal{E}$.
\end{defin}
If $\mathcal E$ is complete, the immersion $(S_{\mathcal E}, {\mathcal E}\cap K)\subset (M,K)$ is a closed immersion of germs and thus $S_{\mathcal E}$ is a closed invariant hypersurface of $M$.

\subsection{Prolongation of Isolated Branches} 
The property we introduce below has been studied in \cite{Mol} for toric type foliations in projective toric surfaces.

In this paper a {\em branch of curve} means an irreducible germ of curve at a point. We mainly consider plane branches, that is, that are contained in two dimensional ambient spaces.

Consider a two-dimensional nonsingular analytic space $S$, a closed irreducible curve $Y\subset S$ and a point $p\in Y$. We say that $Y$ {\em extends } a branch $(\Gamma,p)$ when $(\Gamma,p)$ is one of the irreducible components of the germ of curve $(Y,p)$. Note that if $Y, Y'$ are closed irreducible curves extending $(\Gamma,p)$, we have that $Y'=Y$.

Let $(S,D;\mathcal{G})$ be a foliated surface and take a branch of curve $(\Gamma,p)$ not contained in $D$. We say that $(\Gamma,p)$ is \emph{isolated for $(\mathcal{G},D)$} when for each morphism
$$
\sigma:(S',D'; \mathcal{G}')\rightarrow (S,D; \mathcal{G})
$$
that is the composition of a finite sequence of blowing-ups, we have that  the infinitely near point of $\Gamma$ in $S'$ belongs to the adapted singular locus $\text{Sing}(\mathcal{G}',D')$.
\begin{rem}
Note that an isolated branch $(\Gamma,p)$ is necessarily invariant for $\mathcal{G}$ and that $p \in \text{Sing}(\mathcal{G},D)$.
\end{rem}

\begin{rem}
\label{rem:isolateddesingularized}
When $(S,D; \mathcal{G})$ is desingularized, we have that a given invariant branch $(\Gamma,p)\not\subset (D,p)$ is isolated if and only if $p$ is a singular point of trace type.
\end{rem}

The prolongation property for isolated branches, introduced in next definition, has been studied in \cite{Mol}:
\begin{defin} 
A foliated surface $(S,D;\mathcal{G})$ has the \emph{prolongation property for isolated branches} when, for each isolated branch $(\Gamma,p)$, there is a closed irreducible curve $Y\subset S$ extending $(\Gamma,p)$ such that the branches $(\Upsilon,q) \subset (Y,q)$ are isolated, for each $q\in Y \cap D$.
\end{defin}

\begin{rem} 
\label{lema:prolongacionetapafinal}
Let us assume that the foliated surface $(S,D;{\mathcal G})$ is a complex hyperbolic desingularized foliated surface. Note that in this case ``complex hyperbolic'' exactly means that there are no saddle-nodes in $(S,D;{\mathcal G})$. In this situation, there is a bijection between trace type singularities and isolated branches: we associate to a trace type singularity the only invariant branch through it that is not contained in $D$. In this case the following statements are equivalent:
\begin{enumerate}
\item The foliated surface has the prolongation property for isolated branches.
\item Given a trace type singularity $p$, there is a closed irreducible curve $Y\subset S$ that extends the isolated branch through $p$ and such that $Y\cap D_{\text{dic}}=\emptyset$.
\end{enumerate}
\end{rem}

\subsection{Completeness of Extended Partial Separatrices} 
\label{subs:Extended Partial Separatrices}
As we shall see in the following section, the combinatorial reductions of singularities of local complex hyperbolic foliations in dimension three have the ``prolongation property for isolated branches'' in their restrictions to dicritical components, in the cases that we start with three dicritical components. This allows us to show that the extended partial separatrices are complete, according to next proposition:

\begin{prop} 
\label{prop:prolongacionimplicacompletitud}
Let $(M,E;\mathcal{F})$ be a desingularized complex hyperbolic foliated space of dimension three, where $M$ is a germ along a compact set $K\subset E$. Assume that, for each dicritical component $F$ of $E$, the restricted foliated surface $(F, E|_{F}; \mathcal{F}|_{F})$ has the prolongation property for isolated branches. Then every extended partial separatrix $\mathcal{E}$ is complete.
\end{prop}

\begin{proof} 
Let us recall the immersion of germs $(S_{\mathcal E}, {\mathcal E}\cap K)\subset (M,K)$. The extended partial separatrix $\mathcal E$ is complete if and only if we have that $S_{\mathcal E}\cap E={\mathcal E}$. This property is equivalent to show that
$$
(S_\mathcal{E},q)\cap E=(\mathcal{E},q)
$$
for every $q\in \mathcal{E}\cap K$. Recalling that $e_q(E)$ denotes the number of irreducible components of $E$ through $q$, we have that $1\leq e_q(E)\leq 2$, since there are no corner points in $\mathcal{E}$.

Assume that $e_q(E)=1$. There is a germ of nonsingular vector field that trivializes the foliation and the divisor; the result follows from two-dimensional considerations.

Let us assume now that $e_q(E)=2$. If the irreducible components of $E$ through $q$ are both invariant, we are in the non-dicritical case and we can argue as in \cite{Can-C, Can-R}. Let us suppose now that there is at least one dicritical component $F$ through $q$ and denote by $G$ the other one.

The first observation is that $G$ must be invariant. Let us see this. Assume by contradiction that both $F$ and $G$ are dicritical; in this case, we know that $q$ is a regular point. Recall that ${\mathcal E}$ is a union of irreducible components of ${\mathcal T}_{{\mathcal F}, E}$ and closed irreducible curves $Z$ satisfying the properties in Equation \ref{eq:sigma}. Noting that ${\mathcal T}_{{\mathcal F}, E}$ is contained in the singular locus, there is a $Z$ satisfying the properties in Equation \ref{eq:sigma} such that $q\in Z$. Up to change the role of $F$ and $G$, we can assume $Z \subset F$. We know that there is a point $p \in \text{Sing}(\mathcal{F}) \cap Z$. Note that $e_p(E_{\text{inv}})=1$ and thus it is a trace type singularity for $(F,E\vert_F;{\mathcal F}\vert_F)$. Moreover $(Z,p)$ is the unique isolated branch for $(F,E\vert_F;{\mathcal F}\vert_F)$ at $p$. In particular $Z\subset F$ extends the isolated branch $(Z,p)$. In view of the statement of the ``prolongation property for isolated branches'' in a desingularized situation, stated in Remark \ref{lema:prolongacionetapafinal}, we should have that $Z$ does not intersect the dicritical components for the restricted foliated surface $(F, E|_{F}; \mathcal{F}|_{F})$. This is a contradiction, since $q\in G|_{F}$, which is one of such dicritical components.

Now, we have that $G$ is invariant. There is a closed irreducible curve $Y \subset G$, with $q\in Y$, such that $(Y,q)=(\mathcal{T}_{\mathcal{F},E},q)$. The foliated surface $(F,E|_{F};\mathcal{F}|_{F})$ is desingularized and $q\in \text{Sing}(\mathcal{F}|_{F})$ is of trace type. The unique branch $(\Gamma,q)\not\subset (E|_{F},q)$ invariant for $\mathcal{F}|_{F}$ extends to a  closed irreducible curve $Z\subset F$, by the prolongation property for isolated branches. We have that $Z$ satisfies the conditions in Equation \ref{eq:sigma} and hence $Z\subset {\mathcal E}$. We conclude that $(S_\mathcal{E},q) \cap E=(Y \cup Z,q)=(\mathcal{E},q)$.
\end{proof}

\begin{rem} 
In Lemma \ref{lema:nocompacto} and Corollary \ref{cor:compacto} we show that the prolongation property in the statement of Proposition \ref{prop:prolongacionimplicacompletitud} holds, when we start the reduction of singularities with three dicritical components in the divisor. If we start with only two components, this property is not assured and in this case we can find a linear chain of dicritical components.
\end{rem}

\section{Invariant Surfaces for Toric Type Foliations}
In this section we give an outline of the proof of Theorem \ref{teo:main} that will be completed in Sections \ref{sec:Prolongation Property in Toric Type Foliated Surfaces}, \ref{sec:Corner Type Points under Combinatorial Blowing-ups} and \ref{sec:Totally dicritical case}.

Let us consider a germ of complex hyperbolic foliation $\mathcal{F}_0$ on $(\mathbb{C}^3,0)$, that is of toric type with respect to a normal crossings divisor $E^0$. Recall that we have a combinatorial reduction of singularities
\begin{equation}
\label{eq:sigmasigma}
\sigma:((M,\sigma^{-1}(0)),E;\mathcal{F}) \to ((\mathbb{C}^3,0),E^0;\mathcal{F}_0).
\end{equation}
Let us do some evident reductions of the problem. If $E^0_\text{inv}\ne \emptyset$, there is an invariant surface for $\mathcal{F}_0$ contained in $E^0$ and we are done. On the other hand, if the number of irreducible components of $E^0$ is zero or one, we have that $\sigma$ is the identity morphism, since it is combinatorial. As a consequence, the origin is a simple point and we are also done. Hence, we can assume that $E^0=E^0_\text{dic}$ and that $e_0(E^0)\geq 2$.

We consider two cases for the proof of Theorem \ref{teo:main}: $e_0(E^0)=2$ and $e_0(E^0)=3$.
In both cases, we assume implicitly that all the components of $E^0$ are dicritical.

The proof in the case $e_0(E^0)=3$ is developed in Sections \ref{sec:Prolongation Property in Toric Type Foliated Surfaces} and  \ref{sec:Corner Type Points under Combinatorial Blowing-ups}. The statement proved in that sections is the following one:
\begin{prop} 
\label{prop:propiedadprolongaciontipotorico}
Consider a foliated space $(({\mathbb C}^3,0),E^0;{\mathcal F}_0)$, where ${\mathcal F}_0$ is a complex hyperbolic foliation and $E^0$ has three dicritical irreducible components. Assume that
$$
\sigma: ((M,\sigma^{-1}(0)),E;{\mathcal F})\rightarrow (({\mathbb C}^3,0),E^0;{\mathcal F}_0)
$$
is a combinatorial reduction of singularities. We have:
\begin{enumerate}
\item[a)] Every extended partial separatrix of $(M,E; \mathcal{F})$ is complete.
\item[b)] There is at least one extended partial separatrix of $(M,E; \mathcal{F})$.
\end{enumerate}
\end{prop}

Proposition \ref{prop:propiedadprolongaciontipotorico} gives Theorem \ref{teo:main} in case $e_0(E^0)=3$ as follows. Take an extended partial separatrix $\mathcal E$ assured by part b) of  Proposition \ref{prop:propiedadprolongaciontipotorico}. By the part a) of Proposition
\ref{prop:propiedadprolongaciontipotorico}, we know that $\mathcal E$ is complete and hence we get a closed surface $S_\mathcal{E}\subset M$ invariant for $\mathcal{F}$. Applying Remmert's Proper Mapping Theorem \cite{Rem}, we obtain a surface $\sigma(S_\mathcal{E})$ of $(\mathbb{C}^3,0)$ invariant for $\mathcal{F}_0$, and we are done.

The case $e_0(E^0)=2$ is done by direct arguments in Section \ref{sec:Totally dicritical case}.

\section{Completeness of Extended Partial Separatrices} 
\label{sec:Prolongation Property in Toric Type Foliated Surfaces}

This section is devoted to the proof of Proposition \ref{prop:propiedadprolongaciontipotorico} a). In view of Proposition \ref{prop:prolongacionimplicacompletitud}, it is enough to see that the foliated surface $(F, E|_{F};\mathcal{F}|_{F})$ has the prolongation property for isolated branches, for each dicritical component $F$ of $E$. This fact follows from Lemma \ref{lema:nocompacto} and Corollary \ref{cor:compacto} below.

Let us recall that we are assuming that $e_0(E^0)=3$, that the three irreducible components of $E^0$ are dicritical and that $\sigma$ is a combinatorial reduction of singularities of the complex hyperbolic foliation ${\mathcal F}_0$ as in Equation \ref{eq:sigmasigma}.
\begin{lema}
\label{lema:nocompacto}
Let  $F$ be a non-compact dicritical component of $E$. Then, the foliated surface $(F, E|_{F}; \mathcal{F}|_{F})$ has the prolongation property for isolated branches.
\end{lema}
\begin{proof} 
Let us first show that $F\cap \sigma^{-1}(0)\subset E\vert_F$, where $E\vert_F\subset F$ is the union of the intersections with $F$ of the other components of $E$. Note that $\sigma$ is not the identity morphism, since the intersection of three dicritical components cannot be a simple point. Thus, we have performed at least one blowing-up. In this situation, recalling that $e_0(E^0)=3$ and that we only perform combinatorial blowing-ups, we see that the compact set $\sigma^{-1}(0)$ is a connected union of components of $E$ and compact irreducible curves that are the intersection of two components of $E$. Now, we have to show that given a point $p\in F\cap \sigma^{-1}(0)$ there is an irreducible component $G$ of $E$ with $G\ne F$ and such that $p\in G$. There is a component $G$ of $E$ with $p\in G$ such that either $G$ is a compact component $G\subset \sigma^{-1}(0)$, and hence $G\ne F$ or $G$ is one of the two components of $E$ defining $\sigma^{-1}(0)$ locally at $p$, in this case we can also take $G\ne F$.

Take an isolated branch $(\Gamma,p)$ for $(F, E|_{F};\mathcal{F}|_{F})$, with $p\in F\cap \sigma^{-1}(0)$. We have that $\Gamma\not\subset \sigma^{-1}(0)$, otherwise we would have that $\Gamma\subset E\vert_F$, but this is not possible for an isolated branch. Hence
$$
(\Gamma,p)\subset (F,F\cap \sigma^{-1}(0))
$$ 
is a closed immersion and it extends itself, satisfying in addition that $p$ is the only point in $\Gamma\cap E\vert_F$.
\end{proof}
\begin{lema} 
\label{lema:superficietorica}
Every compact component $F$ of $E$ is a nonsingular projective toric surface, where the restriction $E|_{F}$ is the natural divisor given by the torus action.
\end{lema}
\begin{proof}
Note that an irreducible component $F$ of $E$ is compact if and only if $\sigma(F)=\{0\}$.	
Consider a local coordinate system $(x_1,x_2,x_3)$ at the origin of $\mathbb{C}^3$ such that $E^0=(x_1x_2x_3=0)$. This allows us to give an immersion of $(\mathbb{C}^3,0)$ in $\mathbb{P}_\mathbb{C}^3$ as follows:
$$
(a_1,a_2,a_3)\mapsto [1,a_1,a_2,a_3].
$$
Let $H=H_0 \cup H_1 \cup H_2 \cup H_3$ be the union of the coordinate planes of $\mathbb{P}_\mathbb{C}^3$, in such a way that $H_i \cap (\mathbb{C}^3,0)=(x_i=0)$, for $i=1,2,3$. The projective space $\mathbb{P}_\mathbb{C}^3$ has a structure of toric variety, where $H$ is the divisor provided by the torus action. The combinatorial sequence of blowing-ups $\sigma:((M,\sigma^{-1}(0)),E) \to (({\mathbb C}^3,0),E^0)$ lifts to a combinatorial (equivariant) sequence of blowing-ups
$$
\tilde{\sigma}:(\widetilde{\mathbb{P}}_\mathbb{C}^3,\tilde{H}) \to (\mathbb{P}_\mathbb{C}^3,H).
$$
Each compact irreducible component $F$ of $E$ is an irreducible component of $\tilde{H}$ and we have that $E|_{F}=\tilde{H}|_{F}$. Hence $F$ is a toric surface and the restriction $E|_{F}$ is the divisor defined by the torus action.
\end{proof}
\begin{cor}
\label{cor:compacto}
Let  $F$ be a compact dicritical component of $E$. Then, the foliated surface $(F, E|_{F}; \mathcal{F}|_{F})$ has the prolongation property for isolated branches.
\end{cor}
\begin{proof} 
By Lemma \ref{lema:superficietorica}, we know that $F$ is a nonsingular projective toric surface and $E|_F$ is the normal crossings divisor given by the torus action. In this situation, the results in \cite{Mol} assure that the prolongation property for isolated branches holds for  the desingularized foliated surface $(F,E|_{F};\mathcal{F}|_{F})$.
\end{proof}

The proof of Proposition \ref{prop:propiedadprolongaciontipotorico} a) is ended.
\begin{rem} 
In these results we need the complex hyperbolic hipothesis. More precisely, there are toric type foliations on ${\mathbb P}^2_{\mathbb C}$ with the standard divisor $X_0X_1X_2=0$ that do not satisfy the prolongation property for isolated branches. For instance, if we consider the foliation given in homogeneous coordinates by the $1$-form
$$
(X_0^2-X_0X_2-X_1X_2)\frac{dX_0}{X_0}+(X_0X_2-X_0^2)\frac{dX_1}{X_1}+
X_1X_2\frac{dX_2}{X_2}
$$
we find that it has a combinatorial reduction of singularities, it is not complex hyperbolic and the prolongation property does not hold. The reduction of singularities needs three blowing-ups and the prolongation property falls at the invariant curve $X_0=X_2$; indeed, at the point $[1,0,1]$ it is the isolated branch of a saddle-node, but it is not isolated at the point $[0,1,0]$.
\end{rem}
\section{The Hunt of Trace Singularities} 
\label{sec:Corner Type Points under Combinatorial Blowing-ups}
	
In this section we prove Proposition \ref{prop:propiedadprolongaciontipotorico} b).  In order to do that, we show that the set of trace type singularities $\mathcal{T}_{\mathcal{F},E}$ is not empty. This implies the existence of at least one extended partial separatrix.

Recall that $(({\mathbb C}^3,0), E^0;{\mathcal F}_0)$ is a complex hyperbolic foliated space, where
$e_0(E^0)=3$, the three components of $E^0$ are dicritical and $\sigma$ is a combinatorial reduction of singularities as in Equation \ref{eq:sigmasigma}.

In the proof we use the next version of the ``refined Camacho-Sad's Theorem'' established in \cite{Ort-R-V}:

\begin{quote}
``Let $(S,D;\mathcal{G})\to((\mathbb{C}^2,0),D^0;\mathcal{G}_0)$ be the composition of a finite sequence of blowing-ups, where $\mathcal{G}_0$ is a complex hyperbolic foliation. Assume that there is a connected component $Z$ of $D_\text{inv}$ such that the points of $Z$ are simple for $(S,D;\mathcal{G})$. Then, there is at least one trace type simple singularity in $Z$.''
\end{quote}

The arguments in Lemma \ref{lema:lema1} below are also used in the study of the case $e_0(E^0)=2$ in Section \ref{sec:Totally dicritical case}.
\begin{lema} 
\label{lema:lema1}
If there is an invariant irreducible component $F \subset E_\text{inv}$ such that $\sigma(F)$ is a curve, then $\mathcal{T}_{\mathcal{F},E}\ne \emptyset$.	
\end{lema}

\begin{proof}
Denote $\Gamma= \sigma (F)$. We know that there are two irreducible components $E^0_1$ and $E^0_2$ of $E^0$, such that $\Gamma=E^0_1 \cap E^0_2$. Let us consider the divisor $D \subset E$ defined by
$$
D=\overline{\sigma^{-1}(\Gamma\setminus\{0\})}.
$$	
Denote $E_j$ to the strict transform of $E^0_j$, for $j=1,2$. There is a linear chain $\{D_i\}_{i=0}^{n+1}$ of irreducible components of $E$, with $n \geq 1$, such that $D=\cup_{i=1}^nD_i$, $D_0=E_1$, $D_{n+1}=E_2$ and $D_i\cap D_{i+1}\ne\emptyset$, for every $i\in \{0,1,\ldots,n\}$. Note that $F=D_\ell$ for an index $\ell\in \{1,2,\ldots,n\}$.

Assume $\mathcal{T}_{\mathcal{F},E}=\emptyset$ and let us find a contradiction.
		
Choose a coordinate system $(x_1,x_2,y)$ at the origin of $\mathbb{C}^3$ such that $E_j^0=(x_j=0)$, for $j=1,2$, and  $(y=0)$ is not invariant for $\mathcal{F}_0$. We select a nonzero constant $c\in \mathbb{C}^*$ close enough to the origin. We have that $\Delta^0_c=(y=c)$ is generically transversal to $\mathcal{F}_0$ through the point $q_c=(0,0,c)$. Moreover, the morphism $\sigma$ induces a sequence of blowing-ups between foliated surfaces
$$
(\Delta_c,E|_{\Delta_c} ;\mathcal{F}|_{\Delta_{c}}) \to (\Delta^0_c,E^0|_{\Delta^0_c}; \mathcal{F}_0|_{\Delta^0_{c}}),
$$
where  $\Delta_c$ is the strict transform of $\Delta^0_c$ by $\sigma$. The following remark is key for our arguments:
\begin{quote}
``Let $p$ be a simple corner for a three-dimensional foliated space $(M,E; \mathcal{F})$ and consider a two-dimensional germ $T$ having normal crossings with $E$. We have that $T$ is transversal to $\mathcal{F}$ and that $p$ is also a simple corner point for the restriction $(T,E|_T ;\mathcal{F}|_T)$. Moreover, the point $p$ is singular for $(M,E;\mathcal{F})$ if and only if it is singular for the restriction to $T$.''
\end{quote}
Recalling that $D_\ell$ is invariant and that $D_0,D_{n+1}$ are dicritical components, there are indices $j,k\in \{1,2,\ldots,n\}$ with $j \leq k$ such that $D_i$ is invariant for every $j \leq i \leq k$ and $D_{j-1},D_{k+1}$ are dicritical components. We write $Y_{i}=D_i \cap \Delta_c$, for every $i\in\{0,1,\ldots,n+1\}$.

Note that $Y_{i}$ is an invariant component of $E|_{\Delta_c}$, for every $j\leq i \leq k$. Write $Z=\cup_{i=j}^kY_{i}$; since $\mathcal{T}_{\mathcal{F},E}=\emptyset$, we have that each point in $Z$ is a simple corner for $(M,E;\mathcal{F})$. In view of the above remark, all the points in $Z$ are simple corners for the restriction $(\Delta_c,E|_{\Delta_c}; \mathcal{F}|_{\Delta_{c}})$. Moreover, the points
$$
p_{j-1}=Y_{j-1} \cap Y_{j}, \quad p_{k}=Y_{k} \cap Y_{k+1},
$$
are nonsingular for $\mathcal{F}|_{\Delta_c}$. We conclude that $Y_{j-1}$ and $Y_{k+1}$ must be dicritical components for $(\Delta_c,E|_{\Delta_c} ;\mathcal{F}|_{\Delta_{c}})$. In this way, we find a contradiction with the ``refined Camacho-Sad's Theorem''.
\end{proof}

\begin{proof}[End of the proof of Proposition \ref{prop:propiedadprolongaciontipotorico} b)] 
We recall that it is enough to show that ${\mathcal T}_{{\mathcal F}, E}\ne \emptyset$.
We write $E^0=E^0_1 \cup E^0_2 \cup E^0_3$,  $\Gamma^0_i=E^0_j \cap E^0_k$, with  $\{i,j,k\}=\{1,2,3\}$ and we denote
$$
D^2=\overline{\sigma^{-1}(\Gamma^0_2\setminus\{0\})}, \quad D^3=\overline{\sigma^{-1}(\Gamma^0_3\setminus\{0\})}.
$$
Note that $E=E_1\cup E_2\cup E_3 \cup D^2\cup D^3 \cup \tilde{E}$, where $\tilde{E}=\sigma^{-1}(\Gamma_1^0)$ and $E_i$ is the strict transform of $E^0_i$, for $i=1,2,3$. The restriction $(E_1, E|_{E_1};\mathcal{F}|_{E_1})$ is a desingularized foliated surface, obtained from $(E_1^0,E|_{E_1^0};\mathcal{F}_0|_{E^0_1})$ by a sequence of blowing-ups induced by $\sigma$. We have that
$$
E|_{E_1} =(\tilde{E}\cap E_1) \cup \Gamma_2\cup \Gamma_3,
$$
where $\Gamma_j=D^j\cap E_1$, for $j=2,3$. An irreducible component of $E|_{E_1}$ is invariant for $\mathcal{F}|_{E_1}$ if and only if it is the intersection of $E_1$ with an invariant component of $E$, since $(M,E;\mathcal{F})$ is a desingularized foliated space and $E_1$ is a dicritical component of $E$. In particular, if the branch $\Gamma_3$ is invariant, we have that $\Gamma^0_3$ has been used as center of blowing-up, hence $D^3$ is a normal crossings divisor; moreover, we obtain that $D^3_\text{inv}\ne \emptyset$. In this case, we conclude by Lemma \ref{lema:lema1}. We argue in the same way when $\Gamma_2$ is invariant.
		
Let us suppose now that $\Gamma_2$ and $\Gamma_3$ are not invariant for $\mathcal{F}|_{E_1}$. There are points $p\in E_1$, with $e_{p}(E|_{E_1})=2$. Since $(E_1, E|_{E_1};\mathcal{F}|_{E_1})$ is a desingularized foliated surface, we conclude that
$$
(E|_{E_1})_\text{inv}=\tilde{E}_\text{inv}\cap E_1 \ne \emptyset.
$$
By the ``refined Camacho-Sad's Theorem'', there is a point $p \in \tilde{E}_\text{inv}\cap E_1$ that is a singularity of trace type for $(E_1, E|_{E_1};\mathcal{F}|_{E_1})$. We have that $p$ is also a singularity of trace type for $(M,E;\mathcal{F})$ and we are done.
\end{proof}
The proof of Proposition \ref{prop:propiedadprolongaciontipotorico} is ended. Hence we know that Theorem \ref{teo:main} is true when $e_0(E^0)=3$. It remains to consider the case $e_0(E^0)=2$, that we do in next Section \ref{sec:Totally dicritical case}.
\section{Equirreduction case} 
\label{sec:Totally dicritical case}
We conclude here the proof of Theorem \ref{teo:main} by considering the case $e_0(E^0)=2$. Recall that the two components of $E^0$ are dicritical and $\sigma$ is a combinatorial reduction of singularities as in Equation \ref{eq:sigmasigma}. We look directly for a closed surface $S$ of $M$, invariant by $\mathcal{F}$, with $S\not\subset E$. Its image under $\sigma$ provides the desired invariant surface. The existence of such a closed surface $S$ is given in Proposition \ref{prop:doscomponentes} below.

\begin{prop}
\label{prop:doscomponentes}
Let $\sigma:((M,\sigma^{-1}(0)), E;{\mathcal F})\rightarrow (({\mathbb C}^3,0), E^0;{\mathcal F}_0)$ be a combinatorial reduction of singularities of a complex hyperbolic foliated space, where $E^0$ has two irreducible components, both being dicritical. Then, there is a closed invariant surface 
$$
(S,S\cap\sigma^{-1}(0))\subset (M,\sigma^{-1}(0))
$$ 
such that  $S\not\subset E$.
\end{prop}

The morphism $\sigma$ is a composition of blowing-ups with one-dimensional combinatorial centers. More precisely, let $E^0=E^0_1\cup E^0_2$ be the decomposition of $E^0$ into irreducible components and let $\Gamma$ be the intersection $\Gamma=E^0_1\cap E^0_2$. Taking notations as in Lemma \ref{lema:lema1} we have that $E=E_1\cup E_2\cup D$, where $D=\sigma^{-1}(\Gamma)$ and $E_1,E_2$ are respectively the strict transforms of $E^0_1$ and $E^0_2$.  Moreover, there is a linear chain $\{D_i\}_{i=0}^{n+1}$ of irreducible components of $E$, with $n \geq 0$, such that
$$
D=\cup_{i=1}^nD_i,\quad D_0=E_1, \quad D_{n+1}=E_2
$$
with $D_i\cap D_{i+1}\ne\emptyset$ and $D_i\cap D_{i+t}=\emptyset$, for $i\in \{0,1,\ldots,n\}$ and $t\geq 2$.

We denote $Y_i=D_i\cap \sigma^{-1}(0)$, for $i=1,\ldots,n$. Note that all the $Y_i$ are compact irreducible curves in $E$. Let us define the family ${\mathcal H}$ of subsets $Z\subset \sigma^{-1}(0)$ that are either singletons or connected unions of curves $Y_i$ satisfying the following property:
\begin{quote}
There is a germ of invariant surface $(S_Z,Z)$, with $Z$ as the germification set, that is invariant for $\mathcal F$ and $S_Z\not\subset E$.
\end{quote}
Note that $(S_Z,Z)$ is necessarily unique, since locally it is so. Moreover, if $Z,Z'\in {\mathcal H}$ with $Z'\subset Z$, then $S_{Z'}\subset S_Z$.

\begin{lema}
\label{lema:existence}
${\mathcal H}\ne \emptyset$.
\end{lema}
\begin{proof} 
It is enough to show the existence of a point $p\in \sigma^{-1}(0)$ such that $\{p\}\in {\mathcal H}$. Note that such $\{p\}\in {\mathcal H}$ if and only if $p$ is a trace type simple point, that may be singular or not.

Let us consider first the ``totally dicritical'' case, that is we assume that all the components $D_i$ are dicritical, for $i=0,1,\ldots,n+1$. Any point in $\sigma^{-1}(0)$ is a nonsingular trace type simple point.

Let us consider now the case when there is at least one invariant component in $E$. The arguments in Lemma \ref{lema:lema1} also work in this case and we find a trace type singularity in $\sigma^{-1}(0)$.
\end{proof}
\begin{proof}[Proof of Proposition \ref{prop:doscomponentes}] 
Let us see that there is  $Z\in {\mathcal H}$ such that $(S_Z,Z)\subset (M,\sigma^{-1}(0))$ is a closed  immersion. In view of  Lemma \ref{lema:existence}, we can select an element
$Z'\in {\mathcal H}$. By the local description of simple points, we see that the germ
$$
(S_{Z'}\cap \sigma^{-1}(0), Z')
$$ 
is either $Z'$ or it also contains the germ $(Y_{i_1}\cup Z',Z')$, for $Y_{i_1}\not\subset Z'$. In the first case, we have a closed immersion $(S_{Z'}, Z')\subset (M,\sigma^{-1}(0))$ and we are done. In the second case, we see that $Z_1=Z'\cup Y_{i_1}\in {\mathcal H}$ by local extension of the invariant surface at simple trace points. Repeating the argument, we obtain $Z\supset Z'$, with $Z\in {\mathcal H}$ that defines a closed immersion as desired.
\end{proof}

The proof of Theorem \ref{teo:main} is ended.

\end{document}